\documentclass[12pt]{amsart}
\usepackage{amsmath,latexsym,amscd,amsbsy,amssymb,amsfonts,amsthm,fleqn,leqno,
euscript, graphicx, texdraw, pb-diagram}

\let\SavedRightarrow=\Rightarrow 
\usepackage{marvosym}

\let\Rightarrow=\SavedRightarrow

\newcommand{\COZ}{\operatorname{\textit{coZ}}}

\newcommand{\Seq}{\operatorname{\textit{Seq}}}

\newcommand{\Aaa }{\mathcal A}
\newcommand{\Raa }{\mathcal R}
\newcommand{\Bee }{\mathcal B}
\newcommand{\Cee }{\mathcal C}
\newcommand{\Dee }{\mathcal D}
\newcommand{\Wee }{\mathcal W}
\newcommand{\Vee }{\mathcal V}

\newcommand{\Jee }{\mathbb J}
\newcommand{\Pee }{\mathcal P}

\newcommand{\Uee }{\mathcal U}

\newcommand{\See }{\mathcal S}

\newcommand{\w}{\operatorname{w}}

\renewcommand{\int}{\operatorname{Int}}

\newtheorem{thm}{Theorem}
\newtheorem{pro}[thm]{Proposition}
\newtheorem{lem}[thm]{Lemma}

\newtheorem{cor}[thm]{Corollary}

\begin{document}

\title{Game theoretic approach to   skeletally Dugundji  and Dugundji spaces}

\author{A. Kucharski}
\address{Institute of Mathematics, University of Silesia, Bankowa 14, 40-007 Katowice}
\email{akuchar@ux2.math.us.edu.pl}

\author{Sz. Plewik}
\address{Institute of Mathematics, University of Silesia, Bankowa 14, 40-007 Katowice}
\email{plewik@math.us.edu.pl}

\author{V. Valov}
\address{Department of Computer Science and Mathematics, Nipissing University,
100 College Drive, P.O. Box 5002, North Bay, ON, P1B 8L7, Canada}
\email{veskov@nipissingu.ca}
\thanks{The third author was supported by NSERC Grant 261914-13}



\begin{abstract}
Characterizations of skeletally Dugundji spaces and Dugundji spaces are given in terms of club collections, consisting of countable families of co-zero sets. For example, a Tychonoff space $X$ is skeletally Dugundji  if and only if  there exists  an additive $c$-club on $X$. Dugundji spaces are characterized by the existence of additive $d$-clubs.

\end{abstract}

\maketitle
\markboth{}{Game characterization  of  skeletally  Dugundji spaces}

\section{Introduction}

The aim of this paper is to provide characterizations of Dugundji and skeletally Dugundji spaces in terms of game theory. Recall that Dugundji spaces were introduced in \cite{p} and their topological description as compacta with a multiplicative lattice of open maps was obtained in \cite{s76} and \cite{s81}.
Skeletally Dugundji spaces are skeletal analogue of Dugundji spaces: a Tychonoff space is skeletally Dugundji if it has a multiplicative lattice of skeletal maps, see \cite{kpv1}. This paper is  a continuation of the topics initiated in \cite{kpv} and \cite{kpv1}.

Our key notion of additive clubs consisting of co-zero sets was inspired by condition $(6)$ from \cite[Theorem 5.3.9]{hsh},
 which  characterizes weakly projective Boolean algebras (this condition was attributed to T. Jech).
Together with methods developed in \cite{kp8} this allows  to expand  the study of weakly projective
Boolean algebra (or equivalently its Stone spaces)  to much wider classes, namely Tychonoff spaces.
Another general idea we follow is the possibility to describe some Tychonoff spaces using inverse systems consisting of "nice" bounding maps and second countable spaces.
This method is used in  papers \cite{kp8}, \cite{kp9}  and \cite{kpv}, where "nice" maps are surjective skeletal, open or d-open, and inverse systems are constructed by altering well matched collections of countable families of open sets, which is called club  families. For a certain club family of open sets one can associate a version of the open-open game. This procedure was initiated in \cite{dkz}, where the existence of an appropriate club family  is equivalent to the existence of a winning strategy for player I in the open-open game. It turns out that proofs from \cite{dkz}, \cite{kp8}, \cite{kp9}, \cite{kpv}  and  \cite{va2} dared us to introduce the concept of additive club collection.

The open-open game was introduced by P. Daniels, K. Kunen  and H.~Zhou \cite{dkz}. If a space $X$ is Tychonoff, then we can assume that two players are playing with co-zero sets. Thus, there are two players who take turns playing: a round consists of player I choosing a nonempty set $U\in\COZ(X)$ and player II choosing a nonempty $V\in\COZ(X)$ with
$V\subseteq U$.
 Player I wins if the union of II's open sets is dense in $X$, otherwise player II wins. In the paper \cite{dkz},  there are a few  equivalent  descriptions of the open-open game.
  According to \cite[Theorem 1.6]{dkz}, $X$ is I-favorable if and only if the family $$\{\Pee\in [\COZ(X)]^{\omega}:\Pee\subset_c\COZ(X)\}$$ contains a club, here $\Pee\subset_c\COZ(X)$ means that $\Pee$ is completely embedded in $\COZ(X)$, compare \cite[p. 208]{dkz}.
In fact, this theorem provides a strategy how  player I, knowing the club family, can win the open-open game, and vice versa. Because of this interpretation, the question, whether player I  can use a club family with some additional features, arises naturally.
We describe the cases when player I  constructs a lattice of skeletal (open or $d$-open) map despite the choices of  player II bothers.

If $\Aaa$ is a family of sets, then $\langle \Aaa \rangle$ denotes  the least family which contains $\Aaa$ and it is closed under finite intersections and finite union.
By spaces and maps we mean, respectively, infinite Tychonoff spaces and continuous maps.
In accordance with J. Mioduszewski and L. Rudolf \cite{mr},  a map $f:X\to Y$ is said to be \textit{skeletal}  (resp., \textit{$d$-open} in the sense of M.~Tkachenko, \cite{tk}) if $\int\overline{f[U]}\neq\varnothing$ (resp., $f[U]\subseteq\int\overline{f[U]}$) for every nonempty open $U\subseteq X$. Obviously, every $d$-open map is skeletal.
Moreover, any $d$-open map between  compact Hausdorff spaces  is always open, see \cite{tk}. We say that two maps $f:X\to Y$ and $g:X\to Z$ are {\em homeomorphic} if there is a homeomorphism  $h:Z\to Y$ such that $f=h\circ g$.
Let us also mention that weight of a given space $X$, denoted by $\w(X)$, is always an infinite cardinal. Also, a diagonal product of maps will be briefly called a diagonal.  We recommend the book \cite{eng} and the paper \cite{dkz} for undefined  notions.

\section{Maps constructed using the property $\Seq$}
Let $\Pee$ be a family of open subsets of a topological space $X$. For every $x\in X$ consider the set
$$[x]_{\Pee}=\{y\in X: y\in V \Leftrightarrow x\in V \hbox{~}\text{for all}\hbox{~} V\in \Pee\}.$$ Let $X{/\Pee}$  be  the family of all  classes $[x]_{\Pee}$ and $q_\Pee :X\to X/{\Pee}$ be the map $x \mapsto [x]_\Pee$.
The topology on $X{/\Pee}$ is  generated  by  the  sets $q_\Pee[V]=\{ [x]_\Pee: x \in V \}$, where $V\in \Pee$. 

We have the following fact, see \cite[Lemma 1]{kp8}.

\begin{lem}\label{l4} Let $\Pee$ be a family of open subsets of  a topological space $X$ such that  $X=\bigcup\Pee$. If $\Pee$ is  closed under finite intersections, then    the family $\{q_\Pee[V]: V\in \Pee\}$ is a base for $X{/\Pee}$ and the map $q_\Pee$ is continuous.
\end{lem}

For every family $\Aaa$ of open subsets of $X$ we define $\left\langle \Aaa\right\rangle= \bigcup \{\left\langle \Aaa \right\rangle^n: 0\leq n \}$, where
$\left\langle\Aaa\right\rangle^0=\Aaa$ and $\left\langle\Aaa\right\rangle^{k+1}$ is equal to $$\left\langle\Aaa\right\rangle^{k}\cup\{\{\cup \Vee, \cap \Vee\}: \Vee \subseteq \left\langle \Aaa\right\rangle^k  \mbox{ is a finite family}\}.$$

We get the following.
\begin{thm}\label{t5}
Let $\Pee=\left\langle\bigcup\{\Pee_\alpha:\alpha\in\Lambda\}\right\rangle$ with each $\Pee_\alpha$ being an open cover of $X$ closed under finite intersections and finite unions such that $X{/\Pee_\alpha}$ is a Tychonoff space.
Then $X/\Pee$ is also Tychonoff and the map $q_\Pee$ is homeomorphic to the diagonal map  $q=\triangle\{q_{\Pee_\alpha}:\alpha\in\Lambda\}$.
\end{thm}

\begin{proof} Let $\Aaa=\bigcup\{\Pee_\alpha:\alpha\in\Lambda\}$. Then $\Pee=\bigcup\{\left\langle\Aaa\right\rangle^m: 0\leq m \}$.

\textit{Claim $1$. For any $x,y\in X$ we have $q_\Pee(x)=q_\Pee(y)$ if and only if $q(x)=q(y)$.}

Indeed, since $\Pee_\alpha\subseteq\Pee$, $q_\Pee(x)=q_\Pee(y)$ implies $q_{\Pee_\alpha}(x)=q_{\Pee_\alpha}(y)$ for each $\alpha$. So, $q(x)=q(y)$. Now, suppose $q(x)=q(y)$ for some $x,y\in X$. This yields $q_{\Pee_\alpha}(x)=q_{\Pee_\alpha}(y)$ for each $\alpha$. Hence, $x\in U$ if and only if  $y\in U$ for all $U\in\Aaa$. Using induction, we are going to show that $x\in W$ if and only if  $y\in W$ for all $W\in\left\langle\Aaa\right\rangle^{m}$, where  $0\leq m .$ Suppose this is true for some $k$ and let $W\in\left\langle\Aaa\right\rangle^{k+1}\setminus \left\langle\Aaa\right\rangle^{k}$. Then  there exists a finite family $\{U_1, U_2, \ldots ,U_n\} \subseteq\left\langle\Aaa\right\rangle^{k}$ such that either $W=\bigcup\{U_j: 1\leq j \leq n\}$ or $W=\bigcap \{U_j: 1\leq j \leq n\}$. Because of our inductive assumption, we have $x\in W$ is equivalent to $y\in W$. In this way we conclude that $x\in V$ if and only if  $y\in V$ for all $V\in\Pee$. Consequently, $q_\Pee(x)=q_\Pee(y)$.

So, according to Claim 1, there is a bijective map $h\colon X/\Pee\to q[X]$, $h([x])=(q_{\Pee_\alpha}(x))_{\alpha\in\Lambda}$. We are going to show that $h$ is a homeomorphism. To this end, note that each $Y_\alpha=X/\Pee_\alpha$ is a Tychonoff space and $\Bee_\alpha=\{q_{\Pee_\alpha}[U]:U\in\Pee_\alpha\}$ is a base for the topology of $Y_\alpha$, see Lemma \ref{l4}. So, all set of the form
$$G=\left(q_{\Pee_{\alpha_1}}[U_1]\times \ldots \times q_{\Pee_{\alpha_k}} [U_k]\times\prod_{\alpha \in \Lambda \setminus \{\alpha_i: 1\leq i \leq k\}}  Y_\alpha\right) \cap q[X],$$ where $U_i\in\Pee_{\alpha_i}$,  substitute a base $\Bee$ for $q[X]$. Moreover, each
$$q^{-1}(G)=\bigcap \{q_{\Pee_{\alpha_i}}^{-1}(q_{\Pee_{\alpha_i}}[U_i]): i=1, \ldots, k\} = \bigcap \{ U_i: i=1, \ldots, k\}$$ belongs to $\Pee= \left\langle \Aaa\right\rangle$. Hence, $q_\Pee [q^{-1}(G)]$ is open in $X/\Pee$. This means that $h$ is continuous.

Before proving the final step, observe that $q^{-1}(q[U])=U$ and the image $q[U]$ is open in the space $q[X]$, for all $U\in\Aaa$. Indeed, if $U\in\Pee_\alpha$, then
$q_{\Pee_{\alpha}}[U]$ is open in $Y_\alpha$ and $q_{\Pee_{\alpha}}^{-1}(q_{\Pee_{\alpha}}[U])=U$. Since $\pi_\alpha\circ q=q_{\Pee_{\alpha}}$, where
$\pi_\alpha: q[X]\to Y_\alpha$ denotes the $\alpha$'th projection, we check $$q[U]=\pi_\alpha^{-1}(q_{\Pee_{\alpha}}[U]) \mbox{ and get } q^{-1}(q[U])=q_{\Pee_{\alpha}}^{-1}(q_{\Pee_{\alpha}}[U])=U.$$
\indent
\textit{Claim $2$. $q^{-1}(q[U])=U$ and $q[U]$ is open in $q[X]$, for all $U\in\Pee$.}

In the previous part of the proof, we observed that this is true for all $U\in\Aaa$. We are going to prove by induction this is also true for all $U\in\left\langle\Aaa\right\rangle^{m}$, where  $m>0$. Suppose our statement holds for some $m$. As in the proof of Claim 1, any $W\in\left\langle\Aaa\right\rangle^{m+1}\setminus \left\langle\Aaa\right\rangle^{m}$ is of the form $W=\bigcup\{U_j: 1\leq j \leq n\}$ or $W=\bigcap\{U_j: 1\leq j \leq n\}$, where
$\{U_1, U_2, \ldots ,U_n\}\subseteq\left\langle\Aaa\right\rangle^{m}$. Thus $q[W]=\bigcup\{ q[U_j]: 1\leq j \leq n\}$ or $q[W]=\bigcap \{q[U_j]: 1\leq j \leq n\}$, as   $q^{-1}(q[U_i])=U_i$. Because
each $q[U_i]$ is open in $q[X]$, we have $q[W]$ open in $q[X]$ and $q^{-1}(q[W])=W$.

We can complete the proof of Theorem \ref{t5}. Since $h[q_\Pee[U]]=q[U]$ for any $U\in\Pee$ and $\{q_\Pee[U]:U\in\Pee\}$ is a base for $X/\Pee$,
Claim 2 implies that $h$ is open. Finally, since each $X/\Pee_\alpha$ is a Tychonoff space and $q_\Pee$ is homeomorphic to $q$, $X/\Pee$ is also Tychonoff.
\end{proof}

Note that, the space $X{/\Pee}$ is not always Tychonoff. According to \cite[Theorem 5]{kp8}, the space $X{/\Pee}$ is  Tychonoff if $\Pee$ is closed under finite unions and finite intersections and $\Pee$ has the following property: 

\begin{itemize}
	\item[($\Seq$)] 
For every $W\in\Pee$ there exist sequences  $ \{ U_n: 0\leq n \} \subseteq \Pee \mbox{ and } \{ V_n: 0\leq n \}\subseteq\Pee$ such that $U_k\subseteq X\setminus V_k\subseteq U_{k+1},$  for each $ k$, and $ \bigcup \{ U_n: 0\leq n \}=W.$\end{itemize}

For a  space $X$, let $\COZ(X)$ denote the family of all co-zero subsets of $X$. Recall that X is Tychonoff, so the family $\COZ(X)$ is a base.
 Families of the form $\left\langle\Aaa\right\rangle^n$, which we  use in the proof of the following proposition, have been defined
 after Lemma \ref{l4}.

\begin{pro}\label{pc6}
If  $\{\Aaa_\alpha:\alpha\in\mathbb I\}$ is  a collection of families of co-zero sets such that each family $\Aaa_\alpha$ has the property $\Seq$,  then the family
$\left\langle\bigcup\{ \Aaa_\alpha: \alpha\in \mathbb I\} \right\rangle$ has the property $\Seq$.
\end{pro}

\begin{proof}
We shall show that if the family $\left\langle\bigcup\{ \Aaa_\alpha: \alpha\in \mathbb I\} \right\rangle^{n}$ has the property $\Seq$, then the family
$\left\langle\bigcup\{ \Aaa_\alpha: \alpha\in \mathbb I\} \right\rangle^{n+1}$ has this property, too. For this purpose, the following rules are sufficient. Fix sets $A_1, A_2, \ldots , A_n$ and sequences  $\{ U_k^i: 1\leq k\} \mbox{  and  } \{ V_k^i: 1 \leq k\}$  such that
$$U_k^i\subseteq X\setminus V_k^i\subseteq U_{k+1}^i \mbox{ and } \bigcup\{U_k^i: 1 \leq k\}=A_i,$$ where $1\leq i \leq n$.
By aggregating sequentially the above inclusions, we get $$\bigcup\{U_k^i: 0\leq i \leq n\} \subseteq X\setminus \bigcap\{V_k^i: 1\leq i \leq n\}  \subseteq \bigcup\{U_{k+1}^i: 0\leq i \leq n\}, $$  which also gives
$$\bigcup\{A_i: 0\leq i \leq n \}=\bigcup\{ \bigcup \{ U_k^i: 0\leq i \leq n\}: 1 \leq k \}.$$
Similarly, we get that
$$\bigcap\{U_k^i: 0\leq i \leq n\} \subseteq X\setminus \bigcup\{V_k^i: 1\leq i \leq n\}  \subseteq \bigcap\{U_{k+1}^i: 0\leq i \leq n\}$$  and also
$\bigcap\{A_i: 0\leq i \leq n \}=\bigcup\{ \bigcap \{ U_k^i: 0\leq i \leq n\}: 1 \leq k \}.$ Therefore, we proved that the family $\left\langle\bigcup\{ \Aaa_\alpha: \alpha\in \mathbb I\} \right\rangle^{n+1}$ has the property $\Seq$ provided
$\left\langle\bigcup\{ \Aaa_\alpha: \alpha\in \mathbb I\} \right\rangle^{n}$ has this property. Since,
$\bigcup\{ \Aaa_\alpha: \alpha\in \mathbb I\}$ has the same property, it follows by induction that any $\left\langle\bigcup\{ \Aaa_\alpha: \alpha\in \mathbb I\} \right\rangle^{n}$ has the property $\Seq$. Finally, because $\left\langle\bigcup\{ \Aaa_\alpha: \alpha\in \mathbb I\} \right\rangle$ is
the union of all $\left\langle\bigcup\{ \Aaa_\alpha: \alpha\in \mathbb I\} \right\rangle^{n}$, we conclude that
$\left\langle\bigcup\{\Aaa_\alpha:\alpha\in \mathbb I\} \right\rangle$ has the property $\Seq$.
 \end{proof}
 Every family $\Aaa\subseteq\COZ(X)$ can be extended to a family $\left\langle \Aaa \right\rangle_{\Seq} \subseteq\COZ(X)$ with the following properties:
\begin{itemize}
\item $\left\langle \Aaa \right\rangle_{\Seq} $ has the property $\Seq$;
\item $\left\langle \Aaa \right\rangle_{\Seq} $ is closed under finite unions and finite intersections;
\item $|\left\langle \Aaa \right\rangle_{\Seq} |=|\Aaa|$ provided $\Aaa$ is infinite.
\end{itemize}

 Indeed, for every $W\in\COZ(X)$ fix a continuous function $f_W:X\to [0,1]$ with $W=f_W^{-1}((0,1])$. Let $$U_n^W=f_W^{-1}((1/(n+1),1]) \mbox{ and }  V_n^W=X\setminus f_W^{-1}([1/(n+1),1]),$$ where $n\geq 1$. Let  $\left\langle \Aaa\right\rangle_{\Seq}^0=\Aaa$ and then define by induction the families
$\left\langle\Aaa\right\rangle_{\Seq}^{m+1}$ each one to be the union of $ \left\langle \Aaa\right\rangle_{\Seq}^m$ and all families $\left\langle\Vee\right\rangle$, where $ \Vee \subset \left\langle \Aaa\right\rangle_{\Seq}^m$  is a finite family, and also  $ \{U^W_n: W\in\left\langle \Aaa\right\rangle_{\Seq}^m\}$ and $ \{V^W_n: W\in\left\langle \Aaa\right\rangle_{\Seq}^m\}$, for each $ n\geq 1$. Then the family
$\left\langle \Aaa\right\rangle_{\Seq}=\bigcup\{ \left\langle \Aaa\right\rangle_{\Seq}^m: m\geq 0\}$ is the required extension of $\Aaa$.
\begin{cor}\label{p6}
If  $\{\Aaa_\alpha:\alpha\in\mathbb I\}$ is  a collection of families of co-zero sets,  then the family
$\left\langle\bigcup\{\left\langle\Aaa_\alpha\right\rangle_{\Seq}: \alpha\in \mathbb I\} \right\rangle$ has the property $\Seq$.
\end{cor}

\section{On  multiplicative lattices}
Let $g$ and $\phi$ be two maps defined on a space $X$.   If  there exists a map $h: \phi[X]\to g[X]$  such that $g= h\circ\phi$, then we will briefly write  $\phi\prec g.$
Under this condition, we have $ g^{-1}(U)=\phi^{-1}(h^{-1}(U)) $, for any $U\subseteq g[X]$. Specially, if $U\subseteq g[X]$ is  open, then the  $\phi[g^{-1}(U)] = h^{-1}(U)$ is also open in $\phi[X]$.
Recall, that  a family $\Psi$ of maps with a common domain $X$ is a \textit{multiplicative lattice } (see \cite[Definition 9]{s81}, \cite{va2}, or \cite[p. 1951]{kpv}) whenever:
\begin{itemize}
\item[(L1)] For any map $f\colon X\to f[X]$ there exists $\phi\in\Psi$ with $\phi\prec f$ and $w(\phi[X])\leq w(f[X])$;
\item[(L2)]  If $\{\phi_\alpha:\alpha\in \Jee\}\subset\Psi$, then the diagonal 
$\triangle\{\phi_\alpha:\alpha\in \Jee\}$ is homeomorphic to some element of $\Psi$.
\end{itemize}
If $\Psi$ consists of  maps onto second countable spaces, then the following conditions  are fulfilled:
\begin{itemize}
\item[($\omega$-L1)] For any map $f\colon X\to f[X]$, such that $f[X]$ is second countable,  there exists $\phi\in\Psi$ with $\phi\prec f$;
\item[($\omega$-L2)]  If $\{\phi_n: 0\leq n\}\subset\Psi$, then the diagonal 
$\triangle\{\phi_n:0\leq n\}$ is homeomorphic to some element of $\Psi$,
\end{itemize}
In such a case the family $\Psi$ is called an $\omega$-\textit{multiplicative lattice}.
When a multiplicative ($\omega$-multiplicative) lattice $\Psi$) consists of skeletal maps, it is called a \textit{multiplicative lattice of skeletal maps} (resp. $\omega$-\textit{multiplicative lattice of skeletal maps}).

Denote by $[\COZ(X)]^{ \omega}$ the collection of all countable and infinite families of co-zero sets.  According to  \cite[p. 208]{dkz},  a family $\Cee\subseteq\mathcal [\COZ(X)]^{ \omega}$ is said to be
a \textit{club} if:
\begin{itemize}
	\item[(i)] $\Cee$ is closed under increasing
$\omega$-chains, i.e., if $\Aaa_1\subseteq \Aaa_2\subseteq \ldots$ is an
increasing $\omega$-chain from $\Cee$, then $\bigcup\{ \Aaa_n: 0 < n \} \in\Cee$; \item[(ii)] for any $\Bee\in [\COZ(X)]^{\leq\omega}$ there exists
$\Aaa\in\Cee$ with $\Bee\subseteq \Aaa$,  \end{itemize}
 A club $\Cee$ is said to be a {\em $c$-club} if it satisfies the following condition:
\begin{itemize}
\item[(iii)] If  $\Pee \in \Cee$,  then   $\Pee\subset_c \COZ(X)$, $\Pee$ has the property $\Seq$ and $\Pee$ is closed under finite union and finite intersections.
\end{itemize}
The relationship  $\Pee\subset_c\COZ(X)$  means that $\Pee\subseteq \COZ(X)$ and \begin{itemize}
	\item[(c)]
For any nonempty $V\in\COZ(X)$ there
exists $W\in \Pee$ such that if  $U\in \Pee$ and $U\subseteq W$, then $U\cap
V\neq\varnothing$. \end{itemize}
Let us note that, (c) may be replaced by the following, see \cite[p. 208]{dkz}:
 \begin{itemize}
	\item[(c$^*$)]
For any $\Wee \subseteq \Pee$, the family $\Wee$ is predense in $\Pee$ if and only if $\Wee$ is predense in $\COZ(X)$.  \end{itemize}

Let us explain that a family $\Wee\subseteq \Pee $ is \textit{predense in} $\Pee\subseteq \COZ(X)$, whenever for each $P\in \Pee$ there exist $V\in \Wee$ and $Q\in \Pee$ such that $Q\subseteq V\cap P$. So,  one can  check the equivalence between  conditions (c) and (c$^*$), using that $\Pee$ is closed under finite intersections. Nevertheless, we have the  following explanations.

\begin{lem}\label{rem*} Suppose $\Pee\subset_c \COZ(X)$ is  closed under finite intersections and $\Wee\subseteq\Pee$. If the family $\Wee \subseteq \Pee$ is predense in $\Pee$, then $\Wee$ is predense in $\COZ(X)$.
 \end{lem}
\begin{proof}
Suppose $\Wee$  is not predense in $\COZ(X)$. So, we can fix $A\in \COZ(X)$ such that $A\cap V=\varnothing$ for each $V\in \Wee$. Using  the condition (c), fix $B\in \Pee$ such that if  $B^*\in \Pee$ and $B^*\subseteq B$, then $A\cap B^*\neq\varnothing$. Since $\Wee$ is predense in $\Pee$,  we choose $V\in \Wee$  such that $V\cap B\ne \varnothing$ and $V\cap B  \in \Pee$. Again by (c), we get $A \cap V\cap B\ne \varnothing$, which contradicts  $A \cap V= \varnothing$.\end{proof}

Suppose $ \Pee\subseteq \COZ(X)$  is  closed under finite intersections. If  $\Wee \subseteq \Pee$ and $\Wee$ is predense in $\COZ(X)$, then it has to be predense in $\Pee,$ too. Also, if $\Wee \subseteq \Pee$ and  $\bigcup \Wee$ is dense in $X$, then  $\Wee$ is predense in $\Pee$. In particular, if $\Wee$ is predense in $\COZ(X)$, then $\bigcup \Wee $ has to be dense in $X$.

\begin{lem}\label{rem1}
Suppose $ \Pee\subseteq \COZ(X)$  is  closed under finite intersections. The following conditions are equivalent:
\begin{enumerate}
	\item Any family $\Wee\subseteq \Pee $,  which is predense in $\Pee$, has to be  predense in $\COZ(X)$;
		\item $\Pee \subset_c \COZ(X)$.
\end{enumerate}
 \end{lem}
\begin{proof} In view of Lemma \ref{rem*}, we need to show  the implication $(1)\Rightarrow(2)$. Suppose that (2) is not true. So, we can fix $V\in \COZ(X)$ such that for each $U\in \Pee$ there exists a co-zero set $A \subseteq U$, which satisfies  $A \in \Pee$ and $A\cap V= \varnothing$.  In such a case, the family $$ \{
B\in \Pee: B\cap V =\varnothing\},$$ would be predense in $\Pee$ and not predense in $\COZ(X).$
\end{proof}

If $\Cee$ is a $c$-club and $\left\langle \Aaa_1 \cup \Aaa_2 \right\rangle\in\Cee$ for all   $\Aaa_1, \Aaa_2 \in \Cee$, then $\Cee$ is called an \textit{additive} $c$-club. Every additive $c$-club $\Cee$ has the following property:  $\left\langle \bigcup \{\Aaa_n: 0\leq n\} \right\rangle \in \Cee$ for any   $\{\Aaa_n: 0 \leq  n\} \subset \Cee$. To see this, first check inductively that $\left \langle \bigcup \{A_k: k \leq n \}\right\rangle \in \Cee$, and then use (i).

Now, we show that $\omega$-multiplicative lattices of skeletal maps  implies the existence of additive $c$-clubs.
\begin{thm}\label{888}
If a Tychonoff space   has  an $\omega$-multiplicative lattice of skeletal maps, then there exists  an additive  c-club $\Cee$. If, additionally, $X$  has  a multiplicative lattice of skeletal maps, then $\Cee$ satisfies the following condition: $ \left\langle \bigcup \Raa  \right\rangle $ has the property $\Seq$ and $ \left\langle \bigcup \Raa  \right\rangle \subset_c \COZ(X)$ for any family $\Raa  \subseteq \Cee$.
\end{thm}

\begin{proof}
Suppose  $\Psi_\omega$ is an $\omega$-multiplicative lattice of skeletal maps for the space $X$.
For every $\phi\in\Psi_\omega$ fix an infinite countable  base $\Bee_\phi$, for the space $\phi[X]$ such that $\Bee_\phi$  has the property $\Seq$ and it is closed under finite unions and finite intersections. Let $$\Cee=\{\{\phi^{-1}(V): V \in \Bee_\phi\}:\phi\in\Psi_\omega \}.$$ So, each $\Pee \in \Cee$ is closed under finite union and finite intersections, and it also  has the property $Seq$. Moreover,  we can see that
   $$ (*)\hspace{1cm} \mbox{ If  $\{\Aaa_n: 0 \leq  n\}
\subset\Cee,$ then  $\Aaa=\left\langle \bigcup \{\Aaa_n: 0\leq n\}\right\rangle\in\Cee$.\hspace{1cm}}$$
Indeed,  take $\phi_n \in \Psi_\omega$ such that $\Aaa_n= \{\phi_n^{-1}(V): V \in \Bee_{\phi_n}\}$ for each $n$. Because $\Psi_\omega$ is an $\omega$-multiplicative lattice, condition ($\omega$-L2) implies that
 the diagonal map
 $\phi=\triangle\{\phi_n:0\leq n\}$ belongs to $\Psi_\omega$. Put  $$\Dee_n=\{\pi_n^{-1}( V): V\in \Bee_{\phi_n}\},$$
 where each $\pi_n\colon\phi[X]\to\phi_n[X]$ is the natural projection.  Thus,
 $$\{\phi^{-1}( V): V\in \bigcup \{\Dee_n: 0\leq n \}\}=\bigcup \{\Aaa_n: 0\leq n \}.$$ Let $$\left\{\phi^{-1}( V): V\in \left\langle \bigcup \{\Dee_n: 0\leq n \}\right\rangle\right\} = \Aaa.$$
But  the family $\left\langle\bigcup \{\Dee_n: 0\leq n \}\right\rangle$ is a base for $\phi[X]$ which is closed under finite unions and finite intersections.
Since each $\Dee_n$ has the property $\Seq$, then   the family $\left\langle\bigcup \{\Dee_n: 0\leq n \}\right\rangle$, by Proposition \ref{pc6}, has also the property $\Seq$. Therefore $\Aaa\in\Cee$.

 Condition $(*)$ implies that $\Cee$ is closed under increasing $\omega$-chains. To prove that $\Cee$ satisfies (ii),  suppose $\Aaa\subseteq \COZ(X)$ is a countable family. For every $U\in\Aaa$, there exists a function $g_U:X\to [0,1]$ such that $g_U^{-1}((0,1])=U$ (recall that $U$ is a co-zero set). By  condition ($\omega$-L1), there exists $\phi_U\in\Psi_\omega$ with $\phi_U\prec\ g_U$. The last relation yields that
  $\phi_U[U]$ is open in $\phi_U[X]$ and $\phi_U^{-1}(\phi_U[U])=U$. Since $$\phi_\Aaa=\triangle\{\phi_U:U\in\Aaa\}\in\Psi_\omega,$$ we get
 $\phi_\Aaa^{-1}(\phi_\Aaa[U])=U$ and $\phi_\Aaa[U]$ is open in $\phi_\Aaa[X]$, for all $U\in\Aaa$. Indeed, since $ \pi_U \circ \phi_\Aaa = \phi_U $
 ($\pi_U\colon\phi_{\Aaa}[X]\to\phi_U[X]$ is the projection) we have $$ U= \phi^{-1}_U(\phi_U[U])= \phi^{-1}_\Aaa (\pi^{-1}_U(\phi_U[U])). $$
Thus $\phi_\Aaa [U]= \pi^{-1}_U(\phi_U[U])$ is an open set, and also $U=\phi_\Aaa^{-1}(\phi_\Aaa[U]).$
Let $\Bee_\Aaa \supseteq \{\phi_\Aaa[U]:U\in\Aaa\}$ be a countable base  of the space $\phi_\Aaa[X]$ such that $\left\langle\Bee_\Aaa\right\rangle=\Bee_\Aaa$ and $\Bee_\Aaa $ has the property $\Seq$.   Then   $\{\phi_\Aaa^{-1}(V): V\in \Bee_\Aaa\}\in\Cee$  and contains $\Aaa$, which provides that $\Cee$ satisfies  (ii).

Because each $\phi\in\Psi_\omega$ is a skeletal map, by \cite[Proposition 7]{kp8}, $\Aaa \subset_c \COZ(X)$ for each $\Aaa \in \Cee$. Finally, the additivity
of $\Cee$ follows also from condition $(*)$.

In case $\Psi$ is a multiplicative lattice of skeletal maps for $X$ we define $\Cee$ as above, where $\Psi_\omega \subseteq \Psi$ consists of all possible maps onto second countable spaces.
Consider maps $\phi_\alpha \in \Psi_\omega$ such that $\Aaa_\alpha= \{\phi_\alpha^{-1}(V): V \in \Bee_{\phi_\alpha}\}$ for each $\alpha\in \mathbb{I}$.  The diagonal map
 $\phi=\triangle\{\phi_\alpha : \alpha \in \mathbb{I}\}$ belongs to $\Psi$, so it is a skeletal map.
 Put  $$\Dee_\alpha=\{\pi_\alpha^{-1}( V): V\in \Bee_{\phi_\alpha}\},$$
 where each $\pi_\alpha\colon\phi[X]\to\phi_\alpha[X]$ is the natural projection.  Thus,
 $$\{\phi^{-1}( V): V\in \bigcup \{\Dee_\alpha: \alpha \in \mathbb{I} \}\}=\bigcup \{\Aaa_\alpha: \alpha\in \mathbb{I} \}.$$ Hence,  $$\left\{\phi^{-1}( V): V\in \left\langle \bigcup \{\Dee_\alpha: \alpha \in \mathbb{I} \}\right\rangle\right\} = \left\langle \bigcup \{\Aaa_\alpha: \alpha\in \mathbb{I} \}\right\rangle.$$
The family $\left\langle \bigcup \{\Aaa_\alpha: \alpha\in \mathbb{I} \}\right\rangle$  has the property $\Seq$, due to   Proposition  \ref{pc6}. 
Because the family $\left\langle\bigcup \{\Dee_\alpha: \in \mathbb{I} \}\right\rangle$  is a base for $\phi[X]$, by \cite[Proposition 7]{kp8} we have $\left\langle \bigcup \{\Aaa_\alpha: \alpha\in \mathbb{I} \}\right\rangle \subset_c \COZ(X)$.
\end{proof}

\section{A characterization of skeletally Dugundji spaces}

According to \cite{kpv1} a Tychonoff space $X$ is \textit{skeletally Dugundji if and only if } $X$ has a multiplicative lattice of skeletal maps. Continuing the studies initiated in \cite{kpv1}, we provide a characterization of skeletally Dugundji spaces in terms of $c$-clubs.
First, we establish the next proposition, which is reverse to Theorem \ref{888}.

\begin{pro}\label{t88} Let $X$ be a Tychonoff space $X$.
If there exists  an additive  c-club $\Cee$ such that $ \left\langle \bigcup \Raa  \right\rangle $ has the property $\Seq$ and
$ \left\langle \bigcup \Raa  \right\rangle \subset_c \COZ(X)$ for any family $\Raa  \subseteq \Cee$, then  $X$  has  a multiplicative lattice of skeletal maps.
\end{pro}

\begin{proof} Suppose  $\Cee$ is an additive $c$-club such that $ \left\langle \bigcup \Raa  \right\rangle $ has the property $\Seq$ and
$ \left\langle \bigcup \Raa  \right\rangle \subset_c \COZ(X)$ for any family $\Raa  \subseteq \Cee$.   Let
$\mathbb{S}$ be the collection of all families of the form $\left\langle\bigcup\{\Pee_\alpha:\alpha\in\Lambda\}\right\rangle=\Aaa$, where
$\{\Pee_\alpha:\alpha\in\Lambda\}\subseteq\Cee$. Then, according to Proposition \ref{pc6} each $\Aaa\in\mathbb S$ has the property $\Seq$. Hence, the
space $X/\Aaa$ is Tychonoff and the map $q_\Aaa: X\to X/\Aaa$  is continuous, see \cite[Theorem 5]{kp8}.
 We put
$$\Psi=\{q_\Aaa:\Aaa\in\mathbb{S}\}.$$
  First of all, since $\Aaa\subset_c \COZ(X)$ for all $\Aaa\in\mathbb{S}$, according to \cite[Proposition 7]{kp8}, $\Psi$ consists of skeletal maps. Let us check that $\Psi$ is a multiplicative lattice.
Indeed, suppose that $f\colon X\to f[X]$ is a map. Let $\Bee\subset\COZ(f[X])$ be a base for the topology on $f[X]$ having the least cardinality. Since $\Cee$ is club, for every $V\in\Bee$ there exists $\Pee_V\in\Cee$ with $f^{-1}(V)\in\Pee_V$. Then the family
  $$\Aaa_\Bee=\langle\bigcup\{\Pee_V:V\in\Bee\}\rangle\in\mathbb S$$  contains $\{f^{-1}(V):V\in\Bee\}$.
  Hence,
$q_{\Aaa_\Bee}^{-1}(q_{\Aaa_\Bee}(x))\subseteq f^{-1}(f(x))$, for each $x\in X.$ So, we can  put $h(q_{\Aaa_\Bee}(x))=f(x)$ for each $x\in X$. By Lemma \ref{l4}, the function $h:X/\Aaa_\Bee\to f[X]$ is continuous because  $h^{-1}(V)=q_{\Aaa_\Bee}[f^{-1}(V)]$ is open in $X/\Aaa_\Bee$ for each $V\in\Bee$. Moreover, 
$|\Aaa_\Bee|=|\Bee|$, so $w(q_{\Aaa_\Bee}[X])= w(f[X])$. Thus, (L1) is fulfilled.

To verify (L2), fix $\Jee\subseteq \mathbb{S}$ and consider $\{q_{\Bee}:\Bee\in \Jee\}\subset\Psi$. We put $$\Bee_\Jee=\bigcup\{\left\langle \Bee\right\rangle:\Bee\in\Jee\}=\bigcup\{\Bee:\Bee\in\Jee\}\hbox{~}\mbox{and}\hbox{~}
\Aaa_\Jee=\left\langle \Bee_\Jee\right\rangle.$$ By  Theorem  \ref{t5}, the map $$\triangle\{q_\Bee:\Bee\in \Jee\}: X\to \triangle\{q_\Bee:\Bee\in \Jee\}[X]\subseteq\prod\{X/\Bee: \Bee\in\Jee\}$$ is homeomorphic to the map
$q_{\Aaa_\Jee}:X\to X/\Aaa_\Jee.$ Finally, because $\Aaa_\Jee\in\mathbb{S}$ we have $\triangle\{q_\Bee:\Bee\in \Jee\}\in\Psi$.
\end{proof}

We are now in a position to characterize skeletally Dugundji spaces. The next lemma is used in the proof of that characterization.

\begin{lem}\label{l} If $\Aaa \subseteq\COZ(X)$ and  $\Vee \subseteq  \left\langle \Aaa \right\rangle$ is a finite family, then there exists a finite family $\Wee \subseteq  \Aaa $  such that  $\left\langle \Vee \right\rangle \subseteq \left\langle \Wee \right\rangle \subseteq \left\langle \Aaa\right\rangle.$
 \end{lem}

 \begin{proof} Take $m > 0$ such that $\Vee \subseteq \left\langle \Aaa \right\rangle^{m}.$  If $Y\in \left\langle \Aaa \right\rangle^{q}\setminus \left\langle \Aaa \right\rangle^{q-1}$, where $1\leq q \leq m$,  then there exists an action  $\sigma_n$ and a finite family   $\Uee_Y \in  [\left\langle \Aaa \right\rangle^{q-1}]^{<\omega} $ such that $\sigma_n (\Uee_Y) =Y.$
Let $\Uee_1$ be the union of all families $\Uee_Y$ and $\Vee\cap\left\langle\Aaa\right\rangle^{m-1}$, where $Y\in\Vee\setminus\left\langle\Aaa\right\rangle^{m-1}$. Thus
 $ \Uee_1 \subseteq \left\langle \Aaa \right\rangle^{m-1}$ is a finite family and $\Vee \subseteq \left\langle \Uee_1 \right\rangle. $  Then,  repeat the procedure with $\Uee_1$ instead of $\Vee$, and so on. As a result, we get finite families $ \Uee_i \subseteq \left\langle \Aaa \right\rangle^{m-i}$ such that $$ \Vee \subseteq \left\langle \Uee_1 \right\rangle \subseteq \left\langle \Uee_2 \right\rangle \subseteq \ldots \subseteq \left\langle \Uee_m \right\rangle.$$ Since $\Uee_m \subseteq \Aaa$, the proof is finished.
  \end{proof}

\begin{thm}\label{t8*}
A Tychonoff space $X$ has  a multiplicative lattice of skeletal maps if and only if there exists an additive $c$-club $\Cee\subseteq \COZ(X)$.
\end{thm}
\begin{proof}
According to Theorem \ref{888}, every space with a multiplicative lattice of skeletal maps has an additive $c$-club. To establish the inverse implication, suppose that $X$ has an additive $c$-club $\Cee$. So, if $\Raa  \subseteq \Cee$, then  $\left\langle \bigcup \Raa  \right\rangle $ has the property $\Seq$, because of Proposition \ref{pc6}.
In view of Proposition \ref{t88}, it remains to show that  $\left\langle \bigcup \Raa  \right\rangle \subset_c \COZ(X)$, for each family $\Raa  \subseteq \Cee$. To this end, assume that $\Wee \subseteq \left\langle \bigcup \Raa  \right\rangle$ is predense in $\left\langle \bigcup \Raa  \right\rangle$. Next, choose a countable subfamily $\Wee^* \subseteq \Wee$ such that $\overline{\cup \Wee^*}= \overline{\cup \Wee}$. Because \cite[Theorems 1.1 and 1.6]{dkz}, the space $X$ satisfies the countable chain condition, so the choice of $\Wee^*$ is possible. According to Lemma \ref{l}, for every $W\in\Wee^*$ there exists countably many
elements $\Pee_{\alpha_1}, \ldots ,\Pee_{\alpha_k}$ of $\Raa$ such that $W\in\left\langle\Pee_{\alpha_1}\cup \ldots \cup \Pee_{\alpha_k}\right\rangle$.
So, there exist families $\Pee_n \in \Raa$ such that $$\Wee^* \subseteq\Pee_\Wee=\left\langle\bigcup \{\Pee_n: 0\leq n\}\right\rangle \in \Cee.$$ The family $\Wee^*$ has to be predense in $\left\langle\bigcup \{\Pee_n: 0\leq n\}\right\rangle$. Otherwise, we can find
$V\in \left\langle\bigcup \{\Pee_n: 0\leq n\}\right\rangle$ such that $V \cap \cup \Wee^* = \varnothing$. In this case, $V\cap \overline{\cup \Wee}=\varnothing$. But this contradicts  that $\Pee_\Wee\subseteq\left\langle \bigcup \Raa  \right\rangle$ and 
$\Wee$ is  predense in $\left\langle \bigcup \Raa  \right\rangle$. In view of Lemma \ref{rem*},  the family $\Wee^*$, being predense in $\left\langle\bigcup \{\Pee_n: 0\leq n\}\right\rangle \subset_c \COZ(X)$, has to be predense in $\COZ(X)$. Since predensity of any family in $\COZ(X)$ is equivalent to the density of the family in $X$, $ \Wee$ is also predense in $\COZ(X)$. Finally, by Lemma \ref{rem1},  we get
 $\left\langle \bigcup \Raa  \right\rangle \subset_c \COZ(X)$.
\end{proof}

\begin{cor}\label{c8*}
A Tychonoff space   is skeletally Dugundji  if and only if there exists an additive c-club.
\end{cor}
\begin{cor} A Tychonoff space has $\omega$-multiplicative lattice of skeletal maps if and only if it has multiplicative lattice of skeletal maps.
\end{cor}

\section{Applications with  $d$-open and open maps }

In accordance with \cite{kpv}, we say that a  space $X$ is \textit{$d$-openly generated} if and only if  the family $$\{\Pee\in [\COZ(X)]^{\omega}:\Pee\subset_!\COZ(X)\}$$ contains a club. Here, $\Pee\subset_!\COZ(X)$ means that for any $\See\subset \Pee$ and $x
\not\in\overline{\bigcup\See}$, there exists $W\in\Pee$ such that $x\in W$
and $W\cap\bigcup \See=\varnothing$. It easily seen that
$\Pee\subset_!\COZ(X)$ implies $\Pee\subset_c\COZ(X)$. A club $\Cee$ will be called a
 \textit{$d$-club} provided for each $\Pee \in \Cee$ satisfies the following conditions: $\Pee \subset_! \COZ(X)$, $P$ has the property $\Seq$ and $P$ is closed under finite unions and finite intersections. As in the definition of $c$-clubs,
   a $d$-club  $\Cee$  is said to be \textit{additive d-club}, whenever $\left\langle \Aaa_1 \cup \Aaa_2 \right\rangle\in\Cee$ for all   $\Aaa_1, \Aaa_2 \in \Cee$.
Obviously, each additive $d$-club $\Cee$ has the following property:  $\left\langle \bigcup \{\Aaa_n: 0\leq n\} \right\rangle \in \Cee$ for any   $\{\Aaa_n: 0 \leq  n\} \subset \Cee$.
Furthermore, if $\Psi$ is a  ($\omega$-multiplicative) multiplicative lattice consisting  of  $d$-open  maps on $X$, then we say that $X$ has a \textit{multiplicative lattice of $d$-open  maps} ($\omega$-\textit{multiplicative lattice of $d$-open  maps}). Similarly, we define multiplicative and
$\omega$-multiplicative lattices of  open  maps on X.
Let us note that, a compact Hausdorff space has a multiplicative  lattice of open  maps if and only if it is a Dugundji  space. We  refer to the papers \cite{s81} and \cite{kpv1} for more details concerning Dugundji  and skeletally Dugundji spaces.

\begin{lem}\label{c13}
For any map $f:X\to Y$ the following conditions are equivalent:
\begin{itemize} 
\item[(1)] $f$ is d-open;
\item[(2)] There is a base $\Bee$ of $Y$ such that $\{f^{-1}(V):V\in \Bee\}\subset_!\COZ(X)$;
\item[(3)] $\{f^{-1}(V):V\in \Bee\}\subset_!\COZ(X)$ for every base $\Bee$ of $Y$.
\end{itemize}
\end{lem}

\begin{proof}
According to \cite[Remark 2.2]{kpv}, the first two items are equivalent. Since $(3)$ implies $(2)$, it suffices to show that $(1)$ implies $(3)$.  
So, suppose $f$ is $d$-open. Then, by \cite[Lemma 2.1]{kpv}, 
$\{f^{-1}(V):V\in\mathcal T_Y\}\subset_! \mathcal T_X$, where $\mathcal T_X$ and $\mathcal T_X$ denote the topologies of $X$ and $Y$, respectively. The last relation easily implies that $\{f^{-1}(V): V \in \Bee \} \subset_! \COZ(X)$ for every base $\Bee$ of $Y$.
\end{proof}

Here is the characterization of spaces with lattices of $d$-open maps.
\begin{thm}\label{t9}
For a Tychonoff space $X$ the following conditions are equivalent:
\begin{itemize}
\item[(1)] $X$ has an $\omega$-multiplicative lattice of $d$-open maps;
\item[(2)] There exists an additive $d$-club for $X$;
\item[(3)] $X$ has a multiplicative lattice of $d$-open maps.
\end{itemize}
\end{thm}

\begin{proof}
$(1)\Rightarrow (2)$.
Suppose $X$ has an $\omega$-multiplicative lattice $\Psi_\omega$ of $d$-open maps.  For every $\phi\in\Psi_\omega$ fix an infinite countable  base $\Bee_\phi$ for the space $\phi[X]$ such that $\Bee_\phi$  has the property $\Seq$ and it is closed under finite unions and finite intersections. Since $\phi$ is $d$-open, by Lemma \ref{c13}, we may assume that $\{f^{-1}(V): V \in \Bee \}\subset_! \COZ(X)$.
Let $$\Cee=\{\{\phi^{-1}(V): V \in \Bee_\phi\}:\phi\in\Psi_\omega \}.$$ So, each $\Pee \in \Cee$ is closed under finite union and finite intersections, and it also  has the property $Seq$. As in the proof of Theorem \ref{888}, we can show that $\Cee$ is an additive $d$-club.

 $(2)\Rightarrow (3)$. Suppose there exists an additive $d$-club $\Cee$ on $X$. We need first the following analogue of Proposition \ref{t88}.

\textit{Claim $3$. If  $\left\langle\bigcup\Raa\right\rangle \subset_! \COZ(X)$ for any family $\Raa\subseteq \Cee$, then  $X$  has  a multiplicative lattice of $d$-open maps}.

Denote by $\mathbb{S}$ the collection of all families of the form $\left\langle\bigcup\{\Pee_\alpha:\alpha\in\Lambda\}\right\rangle=\Aaa$, where
$\{\Pee_\alpha:\alpha\in\Lambda\}\subseteq\Cee$. Then, according to Proposition \ref{pc6}, each $\Aaa\in\mathbb S$ has the property $\Seq$. So, the
space $X/\Aaa$ is Tychonoff and the map $q_\Aaa: X\to X/\Aaa$  is continuous. Moreover $\mathcal B_\Aaa=\{q_\Aaa[U]:U\in\Aaa\}$ is a base for $X/\Aaa$
and $q_\Aaa^{-1}(q_\Aaa[U])=U$ for all $U\in\Aaa$.
Let $\Psi=\{q_\Aaa:\Aaa\in\mathbb{S}\}.$
Since $\Aaa=q_\Aaa^{-1}(\mathcal B_\Aaa)\subset_!\COZ(X)$ for all $\Aaa\in\mathbb{S}$, according to Lemma \ref{c13}, $\Psi$ consists of $d$-open maps. Next,
following the proof of Proposition \ref{t88}, we can show that $\Psi$ is a multiplicative lattice on $X$. This provides the proof of Claim 3.

So, according to Claim 3, it suffices to prove that $\left\langle\bigcup\Raa\right\rangle \subset_! \COZ(X)$ for any family $\Raa\subseteq \Cee$. To do this, let $\Wee\subseteq\left\langle\bigcup\Raa\right\rangle$ and $x\not\in\overline{\bigcup\Wee}$. Since $X$ satisfies the countable chain condition (recall that $\Cee$ is also a $c$-club on $X$, so $X$ is an I-favorable space), there is a countable subfamily $\Wee^*\subseteq\Wee$ with
$\overline{\bigcup\Wee^*}=\overline{\bigcup\Wee}$. So, $x\not\in\overline{\bigcup\Wee^*}$. As in the proof of Theorem \ref{t8*}, choose countably many $\Pee_n\in\Raa$ such that  $\Wee^*\subseteq\left\langle\bigcup\{\Pee_n:n\geq 0\}\right\rangle\subseteq\left\langle\bigcup\Raa\right\rangle$. Because $\Pee_\Wee=\left\langle\bigcup\{\Pee_n:n\geq 0\}\right\rangle\in\Cee$, $\Pee_\Wee\subset_!\COZ(X)$. Consequently, there exists $V\in\Pee_\Wee$ with $x\in V$ and
$V\cap\overline{\bigcup\Wee^*}=\varnothing$. This yields that $V\cap\overline{\bigcup\Wee}=\varnothing$. Since $V\in\left\langle\bigcup\Raa\right\rangle$, we finally obtain that $\left\langle\bigcup\Raa\right\rangle \subset_!\COZ(X)$.

The implication $(3)\Rightarrow (1)$ is trivial.
\end{proof}

\begin{cor}
A compact Hausdorff space $X$ is a Dugundji space if and only if  there exists an additive $d$-club for $X$.
\end{cor}
\begin{proof} Since compact spaces having a multiplicative lattice of open maps are exactly Dugundji spaces, the corollary follows directly from Theorem \ref{t9} because of every $d$-open map between  compact Hausdorff spaces  is open, see \cite{tk}. \end{proof}

\end{document}